\documentclass{llncs}
\usepackage[T1]{fontenc}
\usepackage{graphicx,graphbox}
\usepackage[all]{xy}

\usepackage{amsmath,amssymb,mathtools}
\usepackage{booktabs}
\usepackage{mathpartir}
\usepackage{tikz-cd}
\usepackage{enumerate}

\usepackage{pifont}

\usepackage{bbding}

\newcommand{\application}{\bullet} 

\usepackage{url}
\newcommand{\seq}{,\ldots,}

\allowdisplaybreaks

\begin{document}
\title{On Models of the Planar Lambda Calculus}

\author{Chad Nester\thanks{This work was supported by Estonian Research Council grant PRG2764.}}
\institute{University of Tartu\\\email{nester@ut.ee}}


\maketitle             
\begin{abstract}
  We construct an adjunction relating two approaches to modelling the planar $\lambda$-calculus: semi-closed operads and planar $\lambda$-models. We use this to obtain a planar version of Scott's representation theorem.

\end{abstract}

\section{Introduction}

In a relatively recent paper, Hyland has demonstrated that semi-closed cartesian operads are a good notion of model of the (untyped) $\lambda$-calculus, which gives the initial such operad ~\cite{Hyland2017}. Subsequently, Hasegawa has observed that this is also true of (not necessarily cartesian) semi-closed operads and the planar $\lambda$-calculus~\cite{Hasegawa2023Planar}.

A more classical approach to modelling the $\lambda$-calculus is to proceed via combinatory logic. Therein, a model of the $\lambda$-calculus, called a \emph{$\lambda$-model}, is defined to be a combinatory algebra satisfying a number of additional axioms. The additional axioms ensure that the interpretation of $\lambda$-terms as elements of the combinatory algebra in question (via bracket abstraction) respects the $\beta$-equational theory of the $\lambda$-calculus, which need not be the case in general (see e.g.,~\cite{Barendregt1981,Hindley2008}).

In this paper, we consider models of the \emph{planar} $\lambda$-calculus. Specifically, we introduce a notion of \emph{planar $\lambda$-model} and compare it to notion of model given by semi-closed operads. We show that the relationship between these two notions of model is captured by a coreflective adjunction (Theorem~\ref{thm:s-triples-semi-closed}). Essentially, the only difference between a planar $\lambda$-model and a semi-closed operad is that while a planar $\lambda$-model exists in the context of an ambient multicategory, its corresponding semi-closed operad is itself a multicategory that exists at the same metatheoretical level as the ambient multicategory.

We apply this machinery to obtain a version of Scott's representation theorem~\cite{Scott1980} for semi-closed operads, which is to say that we obtain a version of the representation theorem for the planar $\lambda$-calculus (Theorem~\ref{thm:planar-scott}). As an intermediate step, we show that the idempotent splitting completion of a semi-closed multicategory is always a closed multicategory (Lemma~\ref{lem:split-closed}), isolating what would seem to be a fundamental part of the reason that such theorems hold.

Our work here is based on the approach to (substructural) combinatory algebras of Kuzmin et. al.~\cite{Kuzmin2026,Kuzmin2026Journal}. Therein, the structures that we will call \emph{planar combinatory algebras} in this paper were shown to correspond to \emph{weakly-closed} operads. In particular, our adjunction relating planar $\lambda$-models and semi-closed operads is modelled on a more general adjunction~\cite[Theorem 6.12]{Kuzmin2026Journal} relating (substructural) combinatory complete applicative systems and weakly-closed operads, and we follow the lead of~\cite{Kuzmin2026Journal} in developing planar $\lambda$-models as applicative systems in an arbitrary multicategory. We emphasize that a more classical, set-based approach to this sort of thing can be recovered by working in the multicategory of sets and functions.

The classical notion of $\lambda$-model builds on an intermediate notion of $\lambda$-algebra, which itself builds on the notion of combinatory algebra (see e.g., \cite{Barendregt1981,Hindley2008}). Similarly, our notion of planar $\lambda$-model builds on the notion of planar $\lambda$-algebra given by Hasegawa~\cite{Hasegawa2023Planar}, which itself builds on the notion of planar combinatory algebra. We note that what are herein called planar combinatory algebras were originally studied by Tomita~\cite{Tomita2021,Tomita2022,Tomita2023}, who referred to these structures as \emph{$\mathsf{BI}(-)^\bullet$-algebras}. In the terminology of Kuzmin et. al.~\cite{Kuzmin2026Journal}, these structures are referred to as \emph{flipped $\mathsf{BI}$-algebras}. 

Other related work includes the work of Hasegawa on extensional applicative systems and closed operads~\cite{Hasegawa2023,Hasegawa2024}, and the work on Cockett and Hofstra on Turing categories~\cite{Cockett2008}. For a reference on multicategories see Leinster~\cite{Leinster2004}, and for the connection between multicategories and sequent calculus see e.g., Shulman~\cite{Shulman2016} or Cockett and Seely~\cite{Cockett2017}. For closed multicategories see Manzyuk~\cite{Manzyuk2012}.

This paper is organised as follows: In Section~\ref{sec:background} we review multicategories, their connection to sequent calculus, and the various notions of closed multicategory involved in our development. In Section~\ref{sec:recall-kuzmin} we define planar combinatory algebras and recall the adjunction relating them to weakly-closed operads (Theorem~\ref{thm:w-triples-weakly-closed}). Section~\ref{sec:planar-adjunction} we define planar $\lambda$-models and construct an analogous adjunction relating them to semi-closed operads (Theorem~\ref{thm:s-triples-semi-closed}). In Section~\ref{sec:planar-scott} we prove a planar version of Scott's representation theorem (Theorem~\ref{thm:planar-scott}), and in Section~\ref{sec:conclusion} we conclude and discuss a few directions for future work.

\section{Multicategories and Notions of Closure}\label{sec:background}
In this section we review the basic multicategorical machinery that will be required in our development. In particular, this section introduces the notions of weakly-closed, semi-closed, and closed multicategory. We will also discuss the correspondence between multicategories and sequent calculus, which provides an alternative notation for working with multicategories.

\subsection{Multicategory Basics}\label{subsec:multicategory-basics}
We begin with the definition of a multicategory:
\begin{definition}[\cite{Leinster2004}]\label{def:multicategory}
A \emph{multicategory} $\mathcal{M}$ consists of the following data:
\begin{itemize}
  \item a set $\mathcal{M}_0$ whose elements are called the \emph{objects} of $\mathcal{M}$.
  \item for each $n \in \mathbb{N}$ and $A_1, \ldots, A_n, B \in \mathcal{M}_0$, 
        a set $\mathcal{M}(A_1, \ldots, A_n; B)$ of \emph{morphisms}. Any $f \in \mathcal{M}(A_1,\ldots,A_n;B)$ is said to \emph{have arity $n$}.
  \item for each $n \in \mathbb{N}$, $A_1,\ldots,A_n,B \in \mathcal{M}_0$, and $\Gamma_1,\ldots,\Gamma_n \in \mathcal{M}_0^*$, a \emph{composition} operation:
        \[
          \mathcal{M}(A_1,\ldots,A_n;B) \times \mathcal{M}(\Gamma_1;A_1) \times \cdots \times \mathcal{M}(\Gamma_n;A_n)
          \stackrel{\circ}{\longrightarrow}
          \mathcal{M}(\Gamma_1,\ldots,\Gamma_n;B)
        \]
        which we will usually write infix as in $f \circ (g_1,\ldots,g_n) = \circ(f,g_1,\ldots,g_n)$.
  \item for each $A \in \mathcal{M}_0$, an \emph{identity} morphism $1_A \in \mathcal{M}(A; A)$.
      \end{itemize}
      This data must be such that:
      \begin{itemize}
  \item composition is associative. That is, we have:
    \begin{align*}
      & f \circ  (g_1 \circ (h_1^1, \ldots, h_1^{k_1}), \ldots, 
      g_n \circ (h_n^1, \ldots, h_n^{k_n})) \\
      &=  (f \circ  (g_1, \ldots, g_n)) \circ (h_1^1, \ldots, h_1^{k_1}, \ldots, h_n^1, \ldots, h_n^{k_n}) 
    \end{align*}
    whenever $f,g_i,h_i^j$ are morphisms for which the composites make sense.
  \item identity morphisms are unital. That is, we have:
    \[
      f \circ (1_{A_1}, \ldots, 1_{A_n}) = f = 1_B \circ f
    \]
    for every $f \in \mathcal{M}(A_1, \ldots, A_n;B)$.
\end{itemize}
\end{definition}
If $A$ is an object of a multicategory $\mathcal{M}$ we write $A^n$ to indicate the sequence $A,\ldots,A$ consisting of $n$ copies of $A$. Similarly, if $f \in \mathcal{M}(\Gamma;A)$ we write $f^n$ to indicate the sequence consisting of $n$ copies of $f$. For example, if $g \in \mathcal{M}(A^n;B)$ then we may write $g \circ (f^n)$ to indicate the composite $g \circ (f,\ldots,f)$. Moreover, we may write $f \circ g$ instead of $f \circ (g)$ when this improves readability. It is important to note that morphisms of a multicategory $\mathcal{M}$ may have arity $0$, in which case their domain is the empty sequence as in $\mathcal{M}(;B)$. As a special case we have:
\begin{definition}\label{def:operad}
  An \emph{operad} is a multicategory with exactly one object.
\end{definition}
When working with an operad $\mathcal{M}$ we will write $*_\mathcal{M}$ or simply $*$ to denote the unique object of $\mathcal{M}$. For hom-sets, we write $\mathcal{M}(n) = \mathcal{M}(*^n;*)$.

There is a natural notion of functor between multicategories:
\begin{definition}\label{def:multifunctor}
  Let $\mathcal{M},\mathcal{N}$ be multicategories. A \emph{functor} $F : \mathcal{M} \to \mathcal{N}$ consists of:
  \begin{itemize}
  \item a mapping $F_0 : \mathcal{M}_0 \to \mathcal{N}_0$.
  \item a mapping
    \[
      F : \mathcal{M}(A_1,\ldots,A_n;B) \to \mathcal{N}(FA_1,\ldots,FA_n;FB)
    \]
    for all $A_1,\ldots,A_n,B \in \mathcal{M}_0$.
  \end{itemize}
  which must preserve composition and identities, as in:
  \begin{mathpar}
    F(f) \circ (F(g_1),\ldots,F(g_n)) = F(f \circ (g_1,\ldots,g_n))

    \text{and}

    F(1_A) = 1_{FA}
  \end{mathpar}
\end{definition}

\subsection{Multicategories and Sequent Calculus}\label{subsec:sequent-calculus}
We briefly recall the connection between multicategories and sequent calculus. Let $\Sigma$ be a (multi-sorted) signature whose operation symbols are typed as in $f : A_1,\ldots,A_n \vdash B$ where $A_1,\ldots,A_n,B$ are generating sorts. Then a \emph{term over $\Sigma$} is a sequent that is derivable according to the following rules of inference:
\begin{mathpar}
  \inferrule*[right=var]{\text{ }}{x:A \vdash x:A}

  \inferrule*[right=op]{(\Gamma_i \vdash t_i : A_i)_{i =1}^{n} \\ (f : A_1 \seq A_n \vdash B) \in \Sigma}{\Gamma_1 \seq \Gamma_n \vdash f(t_1 \seq t_n) : B}
\end{mathpar}

Crucially, for a sequent $\Gamma \vdash t : B$ to be considered well-formed, the context $\Gamma$ must not contain any repeated variables. This means that, for example, when we write $\Gamma_1 \seq \Gamma_n$ it is implied that the variables in the $\Gamma_i$ are disjoint.

Terms over $\Sigma$ form a multicategory, with identity morphisms given by the \textsc{var} rule and with composition given by substitution. More precisely, the composition operation is given by the following admissible inference rule:
\begin{mathpar}
  \inferrule*[right=cut]{(\Gamma_i \vdash t_i : A_i)_{i =1}^n \\ x_1 {:} A_1 \seq x_n {:} A_n \vdash t : B}{\Gamma_1\seq \Gamma_n \vdash t[t_1 \seq t_n/x_1 \seq x_n] : B}
\end{mathpar}

That the axioms of a multicategory hold follows from certain elementary properties of substitution. For example, for any $\Gamma \vdash t : B$ the right-unitality axiom holds as in:
\[
  (\Gamma \vdash x[t/x] : B) = (\Gamma \vdash t : B)
\]
We omit the redundant information when writing such equations, so that for example $\Gamma \vdash x[t/x] = t : B$ is an equivalent way of expressing the above equation.

This multicategory of terms over a signature is in fact the \emph{free} multicategory over that signature, in the sense that this construction gives the left adjoint of an adjunction between a category of signatures and the category of multicategories. The right adjoint maps a multicategory $\mathcal{M}$ to the $\mathcal{M}_0$-sorted signature $\Sigma_\mathcal{M}$ with an operation symbol $f : A_1,\ldots,A_n \vdash B$ for each morphism $f \in \mathcal{M}(A_1,\ldots,A_n;B)$. The counit of the adjunction gives an interpretation of terms over $\Sigma_\mathcal{M}$ as morphisms of $\mathcal{M}$, and quotienting the multicategory of terms over $\Sigma_\mathcal{M}$ by the equations that hold in the image of this interpretation yields a sequent calculus presentation of $\mathcal{M}$. In this way, the above sequent calculus provides an alternative notation for working with multicategories.

\subsection{Notions of Closure}
The following three notions of closed multicategory play an important role in our development. First, we define weakly-closed multicategories as in:
\begin{definition}[\cite{Kuzmin2026Journal}]\label{def:weakly-closed-multicategory}
    A \emph{weakly-closed multicategory} consists of a multicategory $\mathcal{M}$ together with an object $[A,B] \in \mathcal{M}_0$ and a morphism $\mathsf{ev}_{A,B} \in \mathcal{M}([A,B],A;B)$ for each $A,B \in \mathcal{M}_0$ such that for all $\Gamma \in \mathcal{M}_0^*$ there exists a mapping:
  \[
    \begin{tikzcd}
      \mathcal{M}(\Gamma,A;B) \ar[rrr,"{\Lambda^\Gamma_{A,B}}"] &&& \mathcal{M}(\Gamma;[A,B])
    \end{tikzcd}
  \]
  such that for all $f \in \mathcal{M}(\Gamma,A;B)$ we have $\mathsf{ev}_{A,B} \circ (\Lambda^\Gamma_{A,B}(f),1_A) = f$.
\end{definition}
Note that there may, in general, be more than one choice of mappings $\Lambda^\Gamma_{A,B}$. When working with a given weakly-closed multicategory we implicitly fix one such family of mappings. We prefer this approach over the alternative in which the data of a weakly-closed multicategory includes a specific choice of the $\Lambda^\Gamma_{A,B}$. In particular given the objects $[A,B]$ and morphisms $\mathsf{ev}_{A,B}$, we do not want different choices of $\Lambda^\Gamma_{A,B}$ to result in different weakly-closed multicategories.

Second, we define semi-closed multicategories as in:
\begin{definition}[\cite{Hyland2017}]\label{def:semi-closed-multicategory}
    A \emph{semi-closed multicategory} consists of a multicategory $\mathcal{M}$ together with an object $[A,B] \in \mathcal{M}_0$ and a morphism $\mathsf{ev}_{A,B} \in \mathcal{M}([A,B],A;B)$ for each $A,B \in \mathcal{M}_0$ such that for all $\Gamma \in \mathcal{M}_0^*$ there exists a mapping:
    \[
      \begin{tikzcd}
        \mathcal{M}(\Gamma,A;B) \ar[rrr,"{\Lambda^\Gamma_{A,B}}"] &&& \mathcal{M}(\Gamma;[A,B])
      \end{tikzcd}
    \]
      such that for all $f \in \mathcal{M}(\Gamma,A;B)$ we have $\mathsf{ev}_{A,B} \circ (\Lambda^\Gamma_{A,B}(f),1_A) = f$, and moreover for all $\Gamma = A_1,\ldots,A_n \in \mathcal{M}_0^*$, $g \in \mathcal{M}(\Gamma;[A,B])$, and $\{f_i \in \mathcal{M}(\Gamma_i;A_i)\}_{i=1}^n$ we have:
  \begin{equation}\label{eq:semi-closed-naturality}
    \Lambda^{\Gamma_1,\ldots,\Gamma_n}_{A,B}(g \circ (f_1,\ldots,f_n,1_A)) = \Lambda^\Gamma_{A,B}(g) \circ (f_1,\ldots,f_n)
  \end{equation}
\end{definition}
Clearly every semi-closed multicategory is weakly-closed. As with weakly-closed multicategories, when working with a given semi-closed multicategory we implicitly fix a family of mappings $\Lambda^\Gamma_{A,B}$, which must now have the new naturality property (\ref{eq:semi-closed-naturality}) that differentiates semi-closed and weakly-closed multicategories.

Finally, we define closed multicategories as in:
\begin{definition}[\cite{Manzyuk2012}]\label{def:closed-multicategory}
  A \emph{closed multicategory} consists of a multicategory $\mathcal{M}$ together with an object $[A,B] \in \mathcal{M}_0$ and a morphism $\mathsf{ev}_{A,B} \in \mathcal{M}([A,B],A;B)$ for each $A,B \in \mathcal{M}_0$ such that for all $\Gamma \in \mathcal{M}_0^*$ the mapping:
    \[
      \begin{tikzcd}
        \mathcal{M}(\Gamma;[A,B]) \ar[rrr,"{\mathsf{ev}_{A,B} \circ (-,1_A)}"]&&& \mathcal{M}(\Gamma,A;B)
      \end{tikzcd}
    \]
      has an inverse, which is to say that there exists a mapping:
  \[
    \begin{tikzcd}
      \mathcal{M}(\Gamma,A;B) \ar[rrr,"{\Lambda^\Gamma_{A,B}}"] &&& \mathcal{M}(\Gamma;[A,B])
    \end{tikzcd}
  \]
  such that for all $f \in \mathcal{M}(\Gamma,A;B)$ and $g \in \mathcal{M}(\Gamma;[A,B])$ we have:
  \begin{mathpar}
    \mathsf{ev}_{A,B}(\Lambda^\Gamma_{A,B}(f),1_A)  = f

    \text{and}
    
    \Lambda^\Gamma_{A,B}(\mathsf{ev}_{A,B} \circ (g,1_A)) = g
  \end{mathpar}
\end{definition}
When working with weakly-closed, semi-closed, and closed multicategories it is often convenient to omit the superscripts and/or subscripts from $\mathsf{ev}_{A,B}$ and $\Lambda^\Gamma_{A,B}$, and we will do so when it is unlikely to result in any confusion.

We know that every semi-closed multistory is weakly-closed, and in fact every closed multicategory is semi-closed, as in:
\begin{lemma}\label{lem:close-semi-closed}
  Every closed multicategory $\mathcal{M}$ is semi-closed.
\end{lemma}
\begin{proof}
  For all $g \in \mathcal{M}(\Gamma,A;B)$ and $(f_i \in \mathcal{M}(\Gamma_i;A_i))_{i=1}^n$ we have:
  \begin{align*}
    & \Lambda_{A,B}(g \circ (f_1,\ldots,f_n,1_A))
      = \Lambda_{A,B}(\mathsf{ev}_{A,B} \circ (\Lambda_{A,B}(g),1_A) \circ (f_1,\ldots,f_n,1_A))
    \\&= \Lambda_{A,B}(\mathsf{ev}_{A,B} \circ (\Lambda_{A,B}(g) \circ (f_1,\ldots,f_n),1_A))
    = \Lambda_{A,B}(g) \circ (f_1,\ldots,f_n)
  \end{align*}
  as required.
\end{proof}

Finally, we require a notion of functor that preserves closed structure:
\begin{definition}\label{def:closed-functor}
  Let $\mathcal{M},\mathcal{N}$ be weakly-closed multicategories. We will say that a functor $F : \mathcal{M} \to \mathcal{N}$ is \emph{closed} in case for all $A,B \in \mathcal{M}_0$ we have $F[A,B] = [FA,FB]$ and $F(\mathsf{ev}_{A,B}) = \mathsf{ev}_{A,B}$.
\end{definition}
Notice that the notion of closed functor makes sense for all of weakly-closed, semi-closed, and closed multicategories.

\section{Planar Combinatory Algebras and Weak Closure}\label{sec:recall-kuzmin}
In this section we define planar combinatory algebras and recapitulate the construction of the adjunction relating them to weakly-closed multicategories. Following Kuzmin et. al., we work with applicative systems (and planar combinatory algebras) in some ambient multicategory~\cite{Kuzmin2026Journal}. While this is different from the classical set-based approach, the two coincide when the ambient multicategory happens to be that of sets and functions. We begin with applicative systems:
\begin{definition}[\cite{Kuzmin2026Journal}]\label{def:applicative-system}
  An \emph{applicative system} $(A,\bullet)$ in a multicategory $\mathcal{M}$ consists of an object $A \in \mathcal{M}_0$, called the \emph{carrier}, and a morphism $\bullet \in \mathcal{M}(A,A;A)$, called \emph{application}.
\end{definition}

We introduce notation for applying things multiple times in sequence:
\begin{definition}[\cite{Kuzmin2026Journal}]\label{def:iterated-application}
  Let $(A,\bullet)$ be an applicative system in a multicategory $\mathcal{M}$. We define an \emph{iterated application} morphism $\bullet^n \in \mathcal{M}(A^{n+1};A)$ for each $n \in \mathbb{N}$ as in:
  \begin{mathpar}
    \bullet^0 = 1_A

    \text{and}
    
    \bullet^{n+1} = \bullet \circ (\bullet^n,1_A)
  \end{mathpar}
\end{definition}

Note that $\bullet^1 = \bullet \circ (\bullet^0,1_A) = \bullet \circ (1_A,1_A) = \bullet$. Next, The key definition concerning applicative systems is that of computable morphism:
\begin{definition}[\cite{Kuzmin2026Journal}]\label{def:computable-map}
  Let $(A,\bullet)$ be an applicative system in a multicategory $\mathcal{M}$. A morphism $f \in \mathcal{M}(A^n;A)$ is said to be \emph{$(A,\bullet)$-computable} (or simply \emph{computable}) in case there exists some $a \in \mathcal{M}(;A)$ such that $\bullet^n \circ (a,1_A,\ldots,1_A) = f$. In this case we say that $a$ is an \emph{$(A,\bullet)$-representative (or simply a \emph{representative}) of $f$}.
\end{definition}
Note that every $a \in \mathcal{M}(;A)$ is $(A,\bullet)$-computable since $\bullet^0 \circ (a) = 1_A \circ (a) = a$. 

It is often convenient to use the sequent calculus notation discussed in Section~\ref{subsec:sequent-calculus} when working with applicative systems in multicategories. When doing so, we adopt the usual syntactic conventions for working with applicative systems, wherein $\bullet$ is treated as a left-associative infix binary operator, which is usually omitted, so that for example we have $\bullet(x,y) = x \bullet y = xy$ and $xyz = (xy)z$. Moreover, when we are working with an applicative system $(A,\application)$ in a multicategory $\mathcal{M}$, all of the relevant morphisms are elements of $\mathcal{M}(A^n;A)$ for some $n \in \mathbb{N}$. This allows us to omit the types from our sequents, since everything has type $A$. For example we may write $x_1,x_2 \vdash f(x_1,x_2)$ instead of $x_1 : A , x_2 : A \vdash f(x_1,x_2) : A$, which helps to make things less cluttered. We will moreover allow ourselves to use variable names beyond $x_i$. While this can be made fully formal (see e.g.,~\cite{Shulman2016}), we refrain from doing so here. 

For example, the iterated application morphisms $\bullet^n$ of Definition~\ref{def:iterated-application} can be expressed in the sequent calculus notation as in:
\[
  \bullet^n = (x,x_1,\ldots,x_n \vdash xx_1\cdots x_n) \in \mathcal{M}(A^{n+1};A)
\]
Similarly, for $a \in \mathcal{M}(;A)$ to be an $(A,\bullet)$-representative of $f \in \mathcal{M}(A^n;A)$ (Definition~\ref{def:computable-map}) is to have:
\[
  x_1,\ldots,x_n \vdash ax_1\cdots x_n = f(x_1,\ldots,x_n)
\]

Next, a planar combinatory algebra is an applicative system with respect to which certain morphisms of the ambient multicategory are computable:
\begin{definition}\label{def:planar-combinatory-algebra} 
   An applicative system $(A,\bullet)$ in a multicategory $\mathcal{M}$ is a \emph{planar combinatory algebra} in case $\bullet \circ (1_A,\bullet) \in \mathcal{M}(A^3;A)$ and $1_A \in \mathcal{M}(A;A)$ are $(A,\bullet)$-computable, and for all $a \in \mathcal{M}(;A)$ the morphism $\bullet \circ (1_A,a) \in \mathcal{M}(A;A)$ is $(A,\bullet)$-computable.
\end{definition}
When working with a given planar combinatory algebra we fix a representative $\mathsf{B}$ of $\bullet \circ (1_A,\bullet)$, a representative $\mathsf{I}$ of $1_A$, and a representative $a^\bullet$ of $\bullet \circ (1_A,a)$ for each $a \in \mathcal{M}(;A)$. This means that in the sequent calculus notation we have:
\begin{mathpar}
  \text{{\bf [B]}} \,\, x,y,z \vdash \mathsf{B}xyz = x(yz)

  \text{{\bf [I]}} \,\, x \vdash \mathsf{I}x = x

  \text{{\bf [$^\bullet$]}} \,\, x \vdash a^\bullet x = xa
\end{mathpar}

We recall that the computable morphisms of a given planar combinatory algebra always define a weakly-closed operad:
\begin{lemma}[\cite{Kuzmin2026Journal}]\label{lem:computable-maps-operad}
  Let $(A,\bullet)$ be a planar combinatory algebra in a multicategory $\mathcal{M}$. Then there is a weakly-closed operad $R(\mathcal{M},A,\bullet)$ with morphisms $\overline{f} \in R(\mathcal{M},A,\bullet)$ given by morphisms $f \in \mathcal{M}(A^n;A)$, and with composition and identities inherited from $\mathcal{M}$, so that $\overline{f} \circ (\overline{g_1},\ldots,\overline{g_n}) = \overline{f \circ (g_1,\ldots,g_n)}$ and $1_{*} = \overline{1_A}$.
\end{lemma}

Conversely, reflexive objects in weakly-closed multicategories always define planar combinatory algebras. First, recall the notion of reflexive object:
\begin{definition}\label{def:reflexive-object}
  Let $\mathcal{M}$ be a weakly-closed multicategory. A \emph{reflexive object in $\mathcal{M}$} is a triple $(U,s,r)$ consisting of an object $U \in \mathcal{M}_0$ together with morphisms $s \in \mathcal{M}([U,U];U)$ and $r \in \mathcal{M}(U;[U,U])$ such that $r \circ s = 1_{[U,U]}$.
\end{definition}

Now we have:
\begin{lemma}\label{lem:reflexive-weakly-closed}
  Let $\mathcal{M}$ be a weakly-closed multicategory, and let $(U,s,r)$ be a reflexive object in $\mathcal{M}$. Define $\bullet \in \mathcal{M}(U,U;U)$ by $\bullet = \mathsf{ev}_{U,U} \circ (r,1_A)$. Then $(U,\bullet)$ is a planar combinatory algebra in $\mathcal{M}$.
\end{lemma}
\begin{proof}
  Let $\mathsf{I} = s \circ \Lambda_{U,U}(1_U) \in \mathcal{M}(;U)$. Then $\mathsf{I}$ is a $(U,\bullet)$-representative of $1_U$:
  \begin{align*}
    & \bullet \circ (\mathsf{I},1_U)
      = \mathsf{ev}_{U,U} \circ (r,1_U) \circ (s \circ \Lambda_{U,U}(1_U),1_U)
      = \mathsf{ev}_{U,U} \circ (r \circ s \circ \Lambda_{U,U}(1_U),1_U)
      \\&= \mathsf{ev}_{U,U} \circ (1_{[U,U]} \circ \Lambda_{U,U}(1_U),1_U)
      = \mathsf{ev}_{U,U} \circ (\Lambda_{U,U}(1_U),1_U)
      = 1_U
  \end{align*}
  Let $\mathsf{B}_0 = s \circ \Lambda_{U,U}(\bullet \circ (1_U,\bullet)) \in \mathcal{M}(U,U;U)$,  $\mathsf{B}_1 = s \circ \Lambda_{U,U}(\mathsf{B}_0) \in \mathcal{M}(U;U)$, and $\mathsf{B} = s \circ \Lambda_{U,U}(\mathsf{B}_1) \in \mathcal{M}(;U)$. We show that $\mathsf{B}$ is a $(U,\bullet)$-representative of $\bullet \circ (1_U,\bullet)$. Notice that we have:
  \begin{mathpar}
    \bullet \circ (\mathsf{B}_0,1_U) = \bullet \circ (1_U,\bullet)

    \bullet \circ (\mathsf{B}_1,1_U) = \mathsf{B}_0

    \bullet \circ (\mathsf{B},1_U) = \mathsf{B}_1
  \end{mathpar}
  from which it follows that:
  \begin{align*}
    & \bullet^3 \circ (\mathsf{B},1_U,1_U,1_U)
      = \bullet \circ (\bullet \circ (\bullet \circ (\mathsf{B},1_U),1_U),1_U)
      \\&= \bullet \circ (\bullet \circ (\mathsf{B}_1,1_U),1_U)
      = \bullet \circ (\mathsf{B}_0,1_U)
      = \bullet \circ (1_U,\bullet)
  \end{align*}
  as required. Finally, let $a^\bullet = s \circ \Lambda_{U,U}(\bullet \circ (1_U,a)) \in \mathcal{M}(;U)$ for each $a \in \mathcal{M}(;U)$. We show that $a^\bullet$ is a $(U,\bullet)$-representative of $\bullet \circ (1_U,a)$ as in:
  \begin{align*}
    & \bullet \circ (a^\bullet,1_U)
      = \mathsf{ev}_{U,U} \circ (r,1_U) \circ (s \circ \Lambda_{U,U}(\bullet \circ (1_U,a)), 1_U)
    \\&= \mathsf{ev}_{U,U} \circ (\Lambda_{U,U}(\bullet \circ (1_U,a)),1_U)
    = \bullet \circ (1_U,a)
  \end{align*}
  It follows that $(U,\bullet)$ is a planar combinatory algebra.
\end{proof}

Lemma~\ref{lem:reflexive-weakly-closed} is particularly interesting for its effect on a weakly-closed \emph{operad} $\mathcal{M}$ (Definition~\ref{def:operad}). Notice that in this setting we must have $[*,*] = *$, since $*$ is the only object $\mathcal{M}$, and there is no other possible choice for $[*,*]$. This means that $(*,1,1)$ is a reflexive object, and so Lemma~\ref{lem:reflexive-weakly-closed} gives that $(*,\mathsf{ev}_{*,*})$ is a planar combinatory algebra in $\mathcal{M}$. Lemma~\ref{lem:reflexive-weakly-closed} is a modest generalisation of the corresponding result of Kuzmin et. al.~\cite[Lemma 6.5]{Kuzmin2026Journal}, which only considers weakly-closed operads, instead of arbitrary weakly-closed multicategories.

We end this section by recalling the adjunction relating planar combinatory algebras and weakly-closed operads. To that end, we construct a category of planar combinatory algebras in context:
\begin{definition}\label{def:w-triples}
  A \emph{$w$-triple} $(\mathcal{M},A,\bullet)$ consists of a multicategory $\mathcal{M}$ together with a planar combinatory algebra $(A,\bullet)$ in $\mathcal{*}$. A \emph{morphism} of $w$-triples $F : (\mathcal{M},A,\bullet) \to (\mathcal{N},B,\odot)$ consists of a functor $F : \mathcal{M} \to \mathcal{N}$ such that $FA = B$ and $F(\bullet) = \odot$. Write $\mathsf{T}_w$ for the category of $w$-triples and their morphisms.
\end{definition}

Now, as promised, we have:
\begin{theorem}[\cite{Kuzmin2026Journal}]\label{thm:w-triples-weakly-closed}
  There is an adjunction:
    \[
    \begin{tikzcd}
      \mathsf{O}_w \ar[r,phantom,"{\scriptstyle \perp}"] \ar[r,"L", shift left=0.5em] & \mathsf{T}_w \ar[l,"R", shift left=0.5em]
    \end{tikzcd}
  \]
  where $\mathsf{O}_w$ is the category of weakly-closed operads and closed functors. Moreover, the unit of this adjunction is an isomorphism, witnessing $\mathsf{O}_w$ as a full subcategory of $\mathsf{T}_w$.
\end{theorem}

The interesting part of this adjunction is the action of the functors on objects. For the right adjoint, this is given by the construction of Lemma~\ref{lem:computable-maps-operad}. For the left adjoint, one applies Lemma~\ref{lem:reflexive-weakly-closed}. The rest of the construction is routine.

\section{Planar Lambda Models and Semi-Closure}\label{sec:planar-adjunction}
In this section we give a definition of planar $\lambda$-model, and show that the adjunction of Theorem~\ref{thm:w-triples-weakly-closed} restricts to an adjunction relating planar $\lambda$-models to semi-closed operads. Our definition of planar $\lambda$-model relies on the definition of planar $\lambda$-algebra given by Hasegawa~\cite{Hasegawa2023Planar}. Adapted for our purposes, this is:
\begin{definition}[\cite{Hasegawa2023Planar}]\label{def:planar-lambda-algebra}
  An applicative system $(A,\bullet)$ in a multicategory $\mathcal{M}$ is a \emph{planar $\lambda$-algebra} in case it is a planar combinatory algebra in  which the following axioms hold for some choice of $\mathsf{B}$, $\mathsf{I}$, and $a^\bullet$:
  \begin{enumerate}[{\bf [P.1]}]
  \item $\vdash \mathsf{B}(\mathsf{B}(\mathsf{B}\mathsf{I}))\mathsf{B} = \mathsf{B}$.
  \item $\vdash \mathsf{BII} = \mathsf{I}$.
  \item $\vdash \mathsf{BI}a^\bullet = a^\bullet$ for all $a \in \mathcal{M}(;A)$.
  \item $\vdash \mathsf{BB}^\bullet(\mathsf{BB}(\mathsf{BBB})) = \mathsf{B}(\mathsf{BB})\mathsf{B}$.
  \item $\vdash \mathsf{BI}^\bullet\mathsf{B} = \mathsf{BI}$.
  \item $\vdash \mathsf{B}a^{\bullet \bullet}\mathsf{B} = \mathsf{B}(\mathsf{B}a^\bullet)\mathsf{B}$ for all $a \in \mathcal{M}(;A)$.
  \item $\vdash (ab)^\bullet = \mathsf{B}b^\bullet(\mathsf{B}a^\bullet\mathsf{B})$ for all $a,b \in \mathcal{M}(;A)$.
  \end{enumerate}
\end{definition}
When working with a given planar $\lambda$-algebra we fix representatives $\mathsf{B}$ of $\bullet \circ (1,\bullet)$, $\mathsf{I}$ of $1_A$, and $a^\bullet$ of $\bullet \circ (1_A,a)$ for which axioms {\bf [P.1-7]} hold. We will not attempt to explain these axioms here, noting only that they seem to have been derived analogously to Selinger's axioms for (classical) $\lambda$-algebras~\cite{Selinger2002}. Before proceeding, we note that axiom {\bf [P.3]} is redundant, as in: 
\begin{align*}
    \vdash \,\,&\mathsf{BI}a^\bullet
    \overset{\text{[$^\bullet$]}}{=} a^{\bullet \bullet}(\mathsf{BI})
    \overset{\text{[B]}}{=} \mathsf{B}a^{\bullet \bullet}\mathsf{BI}
    \overset{\text{[P.6]}}{=} \mathsf{B}(\mathsf{B}a^\bullet)\mathsf{BI}
    \\& \overset{\text{[B]}}{=} \mathsf{B}a^\bullet(\mathsf{BI})
    \overset{\text{[P.5]}}{=} \mathsf{B}a^\bullet(\mathsf{BI}^\bullet \mathsf{B})
    \overset{\text{[P.7]}}{=} (\mathsf{I}a)^\bullet
    \overset{\text{[I]}}{=} a^\bullet
\end{align*}

We will require a few technical facts about planar $\lambda$-algebras:
\begin{lemma}[\cite{Hasegawa2023Planar}\footnote{While Lemma~\ref{lem:lambda-algebra-facts}.(i-iv) are due to Hasegawa~\cite[Lemma 3]{Hasegawa2023Planar}, Lemma~\ref{lem:lambda-algebra-facts}.(v-vii) are novel.}]\label{lem:lambda-algebra-facts}
  If $(A,\bullet)$ is a planar $\lambda$-algebra in a multicategory $\mathcal{M}$, we have:
  \begin{enumerate}[(i)]
  \item $x \vdash \mathsf{BB}(\mathsf{B}x) = \mathsf{B}(\mathsf{B}(\mathsf{B}x))\mathsf{B}$. 
  \item $x,y \vdash \mathsf{B}(\mathsf{B}xy) = \mathsf{B}(\mathsf{B}x)(\mathsf{B}y)$. 
  \item $x \vdash \mathsf{B}x\mathsf{I} = \mathsf{BI}x$. 
  \item $x,y,z \vdash \mathsf{B}(\mathsf{B}xy)z = \mathsf{B}x(\mathsf{B}yz)$.
  \item $x \vdash \mathsf{B}(\mathsf{BI})(\mathsf{B}x) = \mathsf{B}x$.
  \item $x,y \vdash \mathsf{BI}(\mathsf{B}xy) = \mathsf{B}xy$.
  \item $x \vdash \mathsf{BI}(\mathsf{B}x) = \mathsf{B}x$.
  \end{enumerate}
\end{lemma}
\begin{proof}
  \begin{enumerate}[(i)]
  \item holds as in:
    \begin{align*}
      x \vdash \,\,& \mathsf{BB}(\mathsf{B}x)
        \overset{\text{[B]}}{=} \mathsf{B}(\mathsf{BB})\mathsf{B}x
        \overset{\text{[P.4]}}{=} \mathsf{BB}^\bullet(\mathsf{BB}(\mathsf{BBB}))x
        \overset{\text{[B]}}{=} \mathsf{B}^\bullet(\mathsf{BB}(\mathsf{BBB})x)
        \\& \overset{\text{[$^\bullet$]}}{=} \mathsf{BB}(\mathsf{BBB})x\mathsf{B}
        \overset{\text{[B]}}{=} \mathsf{B}(\mathsf{BBB}x)\mathsf{B}
        \overset{\text{[B]}}{=} \mathsf{B}(\mathsf{B}(\mathsf{B}x))\mathsf{B}
    \end{align*}
  \item $x,y \vdash \mathsf{B}(\mathsf{B}xy)
        \overset{\text{[B]}}{=} \mathsf{BB}(\mathsf{B}x)y
        \overset{\text{(i)}}{=} \mathsf{B}(\mathsf{B}(\mathsf{B}x))\mathsf{B}y
        \overset{\text{[B]}}{=} \mathsf{B}(\mathsf{B}x)(\mathsf{B}y)$
  \item $x \vdash \mathsf{B}x\mathsf{I}
      \overset{\text{[$^\bullet$]}}{=} \mathsf{I}^\bullet(\mathsf{B}x)
      \overset{\text{[B]}}{=} \mathsf{B}\mathsf{I}^\bullet\mathsf{B}x
      \overset{\text{[P.5]}}{=} \mathsf{B}\mathsf{I}x$
  
  \item $x,y,z \vdash \mathsf{B}(\mathsf{B}xy)z
        \overset{\text{[B]}}{=} \mathsf{B}(\mathsf{B}x)(\mathsf{B}y)z
        \overset{\text{(ii)}}{=} \mathsf{B}x(\mathsf{B}yz)$
        
  \item $x \vdash \mathsf{B}(\mathsf{BI})(\mathsf{B}x) \overset{\text{[B]}}{=} \mathsf{B}(\mathsf{B}(\mathsf{B}I))\mathsf{B}x \overset{\text{[P.1]}}{=} \mathsf{B}x$
        
  \item $x,y \vdash \mathsf{B}xy \overset{\text{(v)}}{=} \mathsf{B}(\mathsf{BI})(\mathsf{B}x)y \overset{\text{[B]}}{=} \mathsf{BI}(\mathsf{B}xy)$

  \item $x \vdash \mathsf{BI}(\mathsf{B}x) \overset{\text{(v)}}{=} \mathsf{BI}(\mathsf{B}(\mathsf{BI})(\mathsf{B}x)) \overset{\text{(vi)}}{=} \mathsf{B}(\mathsf{BI})(\mathsf{B}x) \overset{\text{(v)}}{=} \mathsf{B}x$
    
  \end{enumerate}
\end{proof}

We introduce some convenient notational conventions for working with planar $\lambda$-algebras. First, given $a,b \in \mathcal{M}(;A)$ we write $a * b \in \mathcal{M}(;A)$ to denote the \emph{point composite} of $a$ and $b$, defined as in:
\begin{mathpar}
  a * b = \bullet \circ (\bullet \circ (\mathsf{B},a),b)

  \text{i.e.,}

  \vdash a * b = \mathsf{B}ab
\end{mathpar}
Moreover, for each $n$ in $\mathbb{N}$ and $a \in \mathcal{M}(;A)$ we define $a^{*n} \in \mathcal{M}(;A)$ as in:
   \begin{mathpar}
    a^{*0} = \mathsf{I}

    \text{and}

    a^{*0} = a^{*n} * a
  \end{mathpar}
Note that Lemma~\ref{lem:lambda-algebra-facts}.(iv) says precisely that $*$ is associative. We may therefore omit parentheses when working with $*$, since for example $a*b*c$ is unambiguous.
  
We may now give the definition of planar $\lambda$-model, which is as follows:
\begin{definition}\label{def:planar-lambda-model}
  An applicative system $(A,\bullet)$ in a multicategory $\mathcal{M}$ is a \emph{planar $\lambda$-model} in case it is a planar $\lambda$-algebra, and moreover for all $a,b \in \mathcal{M}(;A)$ and all $n \in \mathbb{N}$ we have:
  \[
    (x_1,\ldots,x_n \vdash ax_1 \cdots x_n = bx_1 \cdots x_n) \Rightarrow (\vdash \mathsf{B}^{*n}\mathsf{I}a = \mathsf{B}^{*n}\mathsf{I}b)
  \]
  We refer to this property as \emph{weak multi-extensionality}.
\end{definition}

The detail-oriented reader may at this point object that the notion of weak multi-extensionality in Definition~\ref{def:planar-lambda-model} is much stronger than the classical notion of weak extensionality one finds in e.g.,~\cite{Scott1980,Hindley2008}, so that our notion of planar $\lambda$-model does not really agree with the classical notion of $\lambda$-model. All is in fact well, and this difference will be reconciled in Section~\ref{sec:enough-points}.

We proceed to establish the central property of planar $\lambda$-models: in the context of some planar $\lambda$-model $(A,\bullet)$ in some multicategory $\mathcal{M}$, let $a \in \mathcal{M}(;A)$ and $n \in \mathbb{N}$. Say that $a$ is \emph{$n$-flat} in case $\vdash \mathsf{B}^{*n}\mathsf{I}a = a$. Then we have:

\begin{lemma}\label{def:unique-flat-representative}
  Let $(A,\bullet)$ be a planar $\lambda$-model in $\mathcal{M}$, and suppose that $f \in \mathcal{M}(A^n;A)$ is $(A,\bullet)$-computable. Then there is a unique $n$-flat $(A,\bullet)$-representative of $f$.
\end{lemma}
\begin{proof}
  Notice that:
  \[
    y,x,x_1,\ldots,x_n \vdash \mathsf{B}^{*n}yxx_1 \cdots x_n = y(xx_1 \cdots x_n)
  \]
  and so in particular for all $a \in \mathcal{M}(;A)$ we have:
  \[
    x_1,\ldots,x_n \vdash \mathsf{B}^{*n}\mathsf{I}ax_1 \cdots x_n = ax_1 \cdots x_n
  \]
  It follows that if $a$ represents $f$, so does $\vdash\mathsf{B}^{*n}\mathsf{I}a \in \mathcal{M}(;A)$. Moreover, we have:
  \[
    x_1,\ldots,x_n \vdash \mathsf{B}^{*n}\mathsf{I}ax_1\cdots x_n = ax_1\cdots x_n
  \]
  and so by weak multi-extensionality we have $\vdash \mathsf{B}^{*n}\mathsf{I}(\mathsf{B}^{*n}\mathsf{I}a) = \mathsf{B}^{*n}\mathsf{I}a$, which is to say that $\vdash \mathsf{B}^{*n}\mathsf{I}a$ is $n$-flat. Thus, there is at least one $n$-flat representative of $f$. If $a \in \mathcal{M}(;A)$ and $b \in \mathcal{M}(;A)$ are both $n$-flat representatives of $f$ then we have:
  \[
    x_1,\ldots,x_n \vdash ax_1\cdots x_n = f(x_1,\ldots,x_n) = bx_1\cdots x_n
  \]
  and so by weak multi-extensionality we have that $\vdash \mathsf{B}^{*n}\mathsf{I}a = \mathsf{B}^{*n}\mathsf{I} b$, and then since $a$ and $b$ are $n$-flat we have $\vdash a = \mathsf{B}^{*n}\mathsf{I}a = \mathsf{B}^{*n}\mathsf{I}b = b$. The claim follows.
\end{proof}
Note that, while different choices of $\mathsf{B}$ and $\mathsf{I}$ may result in different notions of $n$-flat, any such choice determines a unique $n$-flat representative of each computable morphism. As such, any choice of $\mathsf{B}$ and $\mathsf{I}$ allows us to prove Lemma~\ref{lem:computable-semi-closed}, which is the point of all this and, notably, is not sensitive to our choice of $\mathsf{B}$ and $\mathsf{I}$.

In the context of a planar $\lambda$-model $(A,\bullet)$, we will write $a \looparrowright f \in \mathcal{M}(A^n;A)$ to indicate that $a \in \mathcal{M}(;A)$ is the unique $n$-flat representative of $f$. Note that $a$ and $f$ determine each other, so that if one also has $a \looparrowright g \in \mathcal{M}(A^n;A)$ then it follows that $f = g$, and similarly if one has $b \looparrowright f$ then it follows that $a = b$.

We use this machinery to obtain a version of Lemma~\ref{lem:computable-maps-operad} for $\lambda$-models:
\begin{lemma}\label{lem:computable-semi-closed}
  Let $(A,\bullet)$ be a planar $\lambda$-model in $\mathcal{M}$. Then the operad $R(\mathcal{M},A,\bullet)$ of Lemma~\ref{lem:computable-maps-operad} is semi-closed.
\end{lemma}
\begin{proof}
  Take $\mathsf{ev} \in R(\mathcal{M},A,\bullet)(2)$ to be the morphism given by $\bullet \in \mathcal{M}(A,A;A)$, so that $\mathsf{ev} = \overline{\bullet}$. For any $\overline{f} \in R(\mathcal{M},A,\bullet)(n+1)$, let $a \looparrowright f$ and define $\Lambda^{A^n}(f) \in \mathcal{M}(A^n;A)$ by $\vdash \mathsf{B}^{*n}\mathsf{I}a \looparrowright \Lambda^{A^n}(f)$. Now define $\Lambda^n(\overline{f}) \in R(\mathcal{M},A,\bullet)(n)$ as in $\Lambda^n(\overline{f}) = \overline{\Lambda^{A^n}(f)}$.

  We have $\mathsf{bullet} \circ (\Lambda^{A^n}(f),1_A) = f$ as in:
  \begin{align*}
    x_1,\ldots,x_n,x \vdash\,\,& (\mathsf{\bullet} \circ (\Lambda^{A^n}(f),1_A))(x_1,\ldots,x_n,x)
    = \Lambda^{A^n}(f)(x_1,\ldots,x_n)x
    \\&= \mathsf{B}^{*n}\mathsf{I}ax_1 \cdots x_nx
    = ax_1\cdots x_nx
    = f(x_1,\ldots,x_n,x)
  \end{align*}
  which gives:
  \begin{align*}
    & \mathsf{ev} \circ (\Lambda^n(\overline{f}),1)
      = \overline{\bullet} \circ (\overline{\Lambda^{A^n}(f)},\overline{1_A})
      = \overline{\bullet \circ (\Lambda^{A^n}(f),1_A)}
      = \overline{f}
  \end{align*}
  and so $R(\mathcal{M},A,\bullet)$ is weakly-closed. We proceed to show that it is semi-closed. To that end, suppose that $\overline{f} \in R(\mathcal{M},A,\bullet)(n+1)$ with $a \looparrowright f \in \mathcal{M}(A^n;A)$ and that $\left(\overline{g_i} \in R(\mathcal{M},A,\bullet)(m_i) \right)_{i=1}^n$ with $\left(b_i \looparrowright g \in \mathcal{M}(A^{m_i};A)\right)_{i=1}^n$. Define $h \in \mathcal{M}(;A)$ by:
  \[
    h \looparrowright f \circ (g_1,\ldots,g_n,1_A) \in \mathcal{M}(A^{p+1};A)
  \]
  where $p = \sum_{i=1}^n m_i$. Note that then $\vdash \mathsf{B}^{*p}\mathsf{I}h \looparrowright \Lambda^{A^p}(f \circ (g_1,\ldots,g_n,1_A)) \in \mathcal{M}(A^p;A)$. Next, define $k \in \mathcal{M}(;A)$ by:
  \[
    k \looparrowright \Lambda^{A^n}(f) \circ (g_1,\ldots,g_n) \in \mathcal{M}(A^p;A)
  \]
  We will repeatedly use the fact that:
  \begin{align*}
    & x_1^1,\ldots,x_n^{m_n} \vdash
      k x_1^1 \cdots x_n^{m_n}
      = \Lambda^{A^n}(f)(g_1(x_1^1,\ldots,x_1^{m_1}),\ldots,g_n(x_n^1,\ldots,x_n^{m_n}))
      \\&= \mathsf{B}^{*n}\mathsf{I}a(b_1x_1^1 \cdots x_1^{m_1}) \cdots (b_nx_n^1 \cdots x_n^{m_n})
      = a(b_1x_1^1 \cdots x_1^{m_1}) \cdots (b_n x_n^1 \cdots x_n^{m_n})
  \end{align*}
  First, it follows that:
  \begin{align*}
    & x_1^1,\ldots,x_n^{m_n},x \vdash h x_1^1 \cdots x_n^{m_n} x 
      = f(g_1(x_1^1,\ldots,x_1^{m_1}),\ldots,g_n(x_n^1,\ldots,x_n^{m_n}),x)
    \\&= a(b_1x_1^1 \cdots x_1^{m_1}) \cdots (b_n x_n^1 \cdots x_n^{m_n})x
    = kx_1^1 \cdots x_n^{m_n} x
  \end{align*}
  and so weak multi-extensionality gives $\vdash \mathsf{B}^{*p+1}\mathsf{I}h = \mathsf{B}^{*p+1}\mathsf{I}k \in \mathcal{M}(;A)$. Using this, we obtain:
  \begin{align*}
    & x_1^1 ,\ldots, x_n^{m_n} \vdash \mathsf{B}^{*p}\mathsf{I}h x_1^1 \cdots x_n^{m_n}
    = hx_1^1 \cdots x_n^{m_n}
    = \mathsf{B}^{*p+1}\mathsf{I}hx_1^1\cdots x_n^{m_n}
    \\&= \mathsf{B}^{*p+1}\mathsf{I}kx_1^1 \cdots x_n^{m_n}
    = \mathsf{B}^{*p}(\mathsf{BI})kx_1^1 \cdots x_n^{m_n}
    = \mathsf{BI}(kx_1^1 \cdots x_n^{m_n})
    \\&= \mathsf{BI}(a(b_1x_1^1 \cdots x_n^{m_n}) \cdots (b_nx_n^1 \cdots x_n^{m_n}))
    \\&= \mathsf{B}^{*n}(\mathsf{BI})a(b_1x_1^1 \cdots x_n^{m_n}) \cdots (b_nx_n^1 \cdots x_n^{m_n})
    \\&= \mathsf{B}^{*n+1}a(b_1x_1^1 \cdots x_n^{m_n}) \cdots (b_nx_n^1 \cdots x_n^{m_n})
    \\&= a(b_1x_1^1 \cdots x_n^{m_n}) \cdots (b_nx_n^1 \cdots x_n^{m_n})
    \\&= kx_1^1\cdots x_n^{m_n}
  \end{align*}
  That is, $k$ and $\mathsf{B}^{*p}\mathsf{I}h$ represent the same morphism of $\mathcal{M}(A^p;A)$, which is to say that $h = k$ and $\Lambda^{A^n}(f) \circ (g_1,\ldots,g_n) = \Lambda^{A^p}(f \circ (g_1,\ldots,g_n,1_A))$. Then we have:
  \begin{align*}
    &\Lambda^n(\overline{f}) \circ (\overline{g_1},\ldots,\overline{g_n})
      = \overline{\Lambda^{A^n}(f)} \circ (\overline{g_1},\ldots,\overline{g_n})
    \\&= \overline{\Lambda^{A^n}(f) \circ (g_1,\ldots,g_n)}
    = \overline{\Lambda^{A^p}(f \circ (g_1,\ldots,g_n,1_A))}
    \\&= \Lambda^p(\overline{f\circ (g_1,\ldots,g_n,1_A)})
    = \Lambda^p(\overline{f} \circ (\overline{g_1},\ldots,\overline{g_n},1))
  \end{align*}
  and so $R(\mathcal{M},A,\bullet)$ is semi-closed.
\end{proof}
Note that Lemma~\ref{lem:computable-semi-closed} explains how, exactly, every planar $\lambda$-model is a model of the planar $\lambda$-calculus (i.e., the initial semi-closed multicategory).

We show that reflexive objects in semi-closed multicategories define $\lambda$-models:
\begin{lemma}\label{lem:semi-closed-reflexive}
  Let $\mathcal{M}$ be a semi-closed multicategory, and let $(U,s,r)$ be a reflexive object in $\mathcal{M}$. Define $\bullet \in \mathcal{M}(U,U;U)$ by $\bullet = \mathsf{ev}_{U,U} \circ (r,1_U)$. Then $(U,\bullet)$ is a planar $\lambda$-model in $\mathcal{M}$.
\end{lemma}
\begin{proof}
  As in the proof of Lemma~\ref{lem:reflexive-weakly-closed}, we take $\mathsf{I} = s \circ \Lambda(1_U)$, $a^\bullet = s \circ \Lambda(\bullet \circ (1_U,a))$ for all $a \in \mathcal{M}(;A)$, and $\mathsf{B} = s \circ \Lambda(B_1)$ where $B_1 = s \circ \Lambda(B_0)$ and $B_0 = s \circ \Lambda(\bullet \circ (1,\bullet))$. The proof of Lemma~\ref{lem:reflexive-weakly-closed} works here as well, and so we have that $(U,\bullet)$ is a planar combinatory algebra.

  We proceed to show that our data satisfies the equations of a planar $\lambda$-algebra. We begin with {\bf [P.1]}, which holds as in:
    \begin{align*}
      & (\vdash \mathsf{B}(\mathsf{B}(\mathsf{BI}))\mathsf{B})
        = \bullet \circ (\bullet (\mathsf{B}, \bullet \circ (\mathsf{B}, \bullet \circ (\mathsf{B},\mathsf{I}))), \mathsf{B})
        = \bullet \circ (\bullet (\mathsf{B}, \bullet \circ (\mathsf{B}, \mathsf{B}_1 \circ \mathsf{I})), \mathsf{B})
      \\&= \bullet \circ (\bullet (\mathsf{B}, \mathsf{B}_1 \circ \mathsf{B}_1 \circ \mathsf{I}), \mathsf{B})
      = \bullet \circ (\mathsf{B}_1 \circ \mathsf{B}_1 \circ \mathsf{B}_1 \circ \mathsf{I}, \mathsf{B})
      = \mathsf{B}_0 \circ (\mathsf{B}_1 \circ \mathsf{B}_1 \circ \mathsf{I},\mathsf{B})
      \\&= s \circ \Lambda(\bullet \circ (1,\bullet)) \circ (\mathsf{B}_1 \circ \mathsf{B}_1 \circ \mathsf{I},\mathsf{B})
      = s \circ \Lambda(\bullet \circ (1,\bullet) \circ (\mathsf{B}_1 \circ \mathsf{B}_1 \circ  \mathsf{I},\mathsf{B},1))
      \\&= s \circ \Lambda(\bullet \circ (\mathsf{B}_1 \circ \mathsf{B}_1 \circ \mathsf{I}, \bullet \circ (\mathsf{B},1)))
      = s \circ \Lambda(\bullet \circ (\mathsf{B}_1 \circ \mathsf{B}_1 \circ \mathsf{I}, \mathsf{B}_1))
      \\&= s \circ \Lambda(\mathsf{B}_0 \circ (\mathsf{B}_1 \circ \mathsf{I} ,\mathsf{B}_1))
      = s \circ \Lambda(s \circ \Lambda(\bullet \circ (\mathsf{B}_1 \circ \mathsf{I}, \bullet \circ (\mathsf{B}_1,1))))
      \\&= s \circ \Lambda(s \circ \Lambda(\mathsf{B}_0 \circ (\mathsf{I}, \mathsf{B}_0)))
      = s \circ \Lambda(s \circ \Lambda(s \circ \Lambda(\bullet \circ (\mathsf{I},\bullet \circ (\mathsf{B}_0,1)))))
      \\&= s \circ \Lambda(s \circ \Lambda(s \circ \Lambda(\bullet \circ (1,\bullet))))
       = s \circ \Lambda(s \circ \Lambda(\mathsf{B}_0))
      = s \circ \Lambda(\mathsf{B}_1)
      = \mathsf{B}
    \end{align*}
For {\bf [P.2]}, we have:
\begin{align*}
  & (\vdash \mathsf{BII})
    = \bullet \circ (s \circ \Lambda(s \circ \Lambda(\bullet)),\mathsf{I})
    = s \circ \Lambda(\bullet) \circ \mathsf{I}
    = s \circ \Lambda(\bullet \circ (\mathsf{I},1_U))
    = s \circ \Lambda(1_U)
    = \mathsf{I}
\end{align*}
For {\bf [P.4]}, we have:
\begin{align*}
  & (\vdash \mathsf{BB}^\bullet(\mathsf{BB}(\mathsf{BBB})))
    = \bullet \circ (\bullet \circ (\mathsf{B},\mathsf{B}^\bullet),\bullet \circ (\bullet \circ (\mathsf{B},\mathsf{B}), \bullet  \circ (\bullet \circ (\mathsf{B},\mathsf{B}),\mathsf{B})))
      \\&= \mathsf{B}_0 \circ (\mathsf{B}^\bullet, \bullet \circ (\bullet \circ (\mathsf{B},\mathsf{B}), \bullet \circ (\bullet \circ (\mathsf{B},\mathsf{B}), \mathsf{B})))
      = \mathsf{B}_0 \circ (\mathsf{B}^\bullet, \bullet \circ (\mathsf{B}_1 \circ \mathsf{B}, \bullet \circ (\mathsf{B}_1 \circ \mathsf{B}, \mathsf{B})))
      \\&= \mathsf{B}_0 \circ (\mathsf{B}^\bullet, \bullet \circ (\mathsf{B}_1 \circ \mathsf{B}, \mathsf{B}_0 \circ (\mathsf{B},\mathsf{B})))
      = \mathsf{B}_0 \circ (\mathsf{B}^\bullet, \mathsf{B}_0 \circ (\mathsf{B}, \mathsf{B}_0 \circ (\mathsf{B},\mathsf{B})))
      \\& = s \circ \Lambda(\bullet \circ (\mathsf{B}^\bullet,\bullet \circ (\mathsf{B}_0 \circ (\mathsf{B}, \mathsf{B}_0 \circ (\mathsf{B},\mathsf{B})),1)))
      = s \circ \Lambda(\bullet \circ (\mathsf{B}^\bullet,\bullet \circ (\mathsf{B},\bullet \circ (\mathsf{B}_0 \circ (\mathsf{B},\mathsf{B}),1))))
      \\&= s \circ \Lambda(\bullet \circ (\mathsf{B}^\bullet, \bullet \circ (\mathsf{B},\bullet \circ (\mathsf{B},\bullet \circ (\mathsf{B},1)))))
      = s \circ \Lambda(\bullet \circ (\mathsf{B}^\bullet, \mathsf{B}_1 \circ \mathsf{B}_1 \circ \mathsf{B}_1))
      \\&= s \circ \Lambda(\bullet \circ (\mathsf{B}_1 \circ \mathsf{B}_1 \circ \mathsf{B}_1, \mathsf{B}))
      = s \circ \Lambda(\mathsf{B}_0 \circ (\mathsf{B}_1 \circ \mathsf{B}_1,\mathsf{B}))
      \\&= s \circ \Lambda(s \circ \Lambda(\bullet \circ (\mathsf{B}_1 \circ \mathsf{B}_1, \bullet \circ (\mathsf{B},1))))
      = s \circ \Lambda(s \circ \Lambda(\bullet \circ (\mathsf{B}_1 \circ \mathsf{B}_1, \mathsf{B}_1)))
      \\&= s \circ \Lambda(s \circ \Lambda(\bullet \circ (\mathsf{B}_1 \circ \mathsf{B}_1, \mathsf{B}_1)))
      = s \circ \Lambda(s \circ \Lambda(\mathsf{B}_0 \circ (\mathsf{B}_1, \mathsf{B}_1)))
      \\&= s \circ \Lambda(s \circ \Lambda(s \circ \Lambda(\bullet \circ (\mathsf{B}_1,\bullet \circ (\mathsf{B}_1,1)))))
      = s \circ \Lambda(s \circ \Lambda(s \circ \Lambda(\mathsf{B}_0 \circ (1,\mathsf{B}_0))))
      \\&= s \circ \Lambda(s \circ \Lambda(s \circ \Lambda(s \circ \Lambda(\bullet \circ (1,\bullet \circ (\mathsf{B}_0,1))))))
      \\&= s \circ \Lambda(s \circ \Lambda(s \circ \Lambda(s \circ \Lambda(\bullet \circ (1,\bullet \circ (1,\bullet))))))
      \\&= s \circ \Lambda(s \circ \Lambda(s \circ \Lambda(s \circ \Lambda( \bullet \circ (1,\bullet) \circ (1,1,\bullet)  ))))
      \\&= s \circ \Lambda(s \circ \Lambda(s \circ \Lambda(s \circ \Lambda( \bullet \circ (\mathsf{B}_0,1) \circ (1,1,\bullet)  ))))
      \\&= s \circ \Lambda(s \circ \Lambda(s \circ \Lambda(s \circ \Lambda( \bullet \circ (1,\bullet) \circ (\mathsf{B}_0,1,1)  ))))
      \\&= s \circ \Lambda(s \circ \Lambda(s \circ \Lambda(s \circ \Lambda( \bullet \circ (1,\bullet) ) \circ (\mathsf{B}_0,1) )))
      = s \circ \Lambda(s \circ \Lambda(s \circ \Lambda(\mathsf{B}_0 \circ (\mathsf{B}_0,1) )))
      \\&= s \circ \Lambda(s \circ \Lambda(s \circ \Lambda(\mathsf{B}_0) \circ \mathsf{B}_0))
      = s \circ \Lambda(s \circ \Lambda(\mathsf{B}_1 \circ \mathsf{B}_0))
      = s \circ \Lambda(s \circ \Lambda(  \bullet \circ (\mathsf{B},\bullet \circ (\mathsf{B}_1,1))  ))
      \\&= s \circ \Lambda(s \circ \Lambda(  \bullet \circ (1,\bullet) \circ (\mathsf{B},\mathsf{B}_1,1)  ))
      = s \circ \Lambda(s \circ \Lambda(  \bullet \circ (1,\bullet) ) \circ (\mathsf{B},\mathsf{B}_1)  )
      \\&= s \circ \Lambda(\mathsf{B}_0 \circ (\mathsf{B},\mathsf{B}_1))
      = s \circ \Lambda(\bullet \circ (\mathsf{B}_1 \circ \mathsf{B},\bullet \circ (\mathsf{B},1)) )
      = s \circ \Lambda(\bullet \circ (1,\bullet) \circ (\mathsf{B}_1 \circ \mathsf{B}, \mathsf{B},1))
      \\&= s \circ \Lambda(\bullet \circ (1,\bullet)) \circ (\mathsf{B}_1 \circ \mathsf{B}, \mathsf{B})
      = \mathsf{B}_0 \circ (\mathsf{B}_1 \circ \mathsf{B}, \mathsf{B})
      = \bullet \circ (\mathsf{B}_1 \circ \mathsf{B}_1 \circ \mathsf{B},\mathsf{B})
      \\&= \bullet \circ (\bullet \circ (\mathsf{B},\mathsf{B}_1 \circ \mathsf{B}),\mathsf{B})
      = \bullet \circ (\bullet \circ (\mathsf{B},\bullet \circ (\mathsf{B},\mathsf{B}),\mathsf{B}))
  = (\vdash \mathsf{B}(\mathsf{BB})\mathsf{B})
\end{align*}
For {\bf [P.5]}, we have:
\begin{align*}
  & (\vdash \mathsf{BI}^\bullet\mathsf{B})
    = \bullet \circ (\bullet \circ (\mathsf{B},\mathsf{I}^\bullet),\mathsf{B})
    = \bullet \circ (\mathsf{B}_1 \circ \mathsf{I}^\bullet,\mathsf{B})
    = \mathsf{B_0} \circ (\mathsf{I}^\bullet,\mathsf{B})
  \\&= s \circ \Lambda(\bullet \circ (\mathsf{I}^\bullet, \bullet \circ (\mathsf{B},1)))
  = s \circ \Lambda(\bullet \circ (\bullet \circ (\mathsf{B},1),\mathsf{I}))
  = s \circ \Lambda(\bullet \circ (\mathsf{B}_1,\mathsf{I}))
  \\&= s \circ \Lambda(\mathsf{B}_0 \circ (1,\mathsf{I}))
  = s \circ \Lambda(s \circ \Lambda(  \bullet \circ (1,\bullet \circ (\mathsf{I},1))  ))
  = s \circ \Lambda(s \circ \Lambda( \bullet ))
  \\&= s \circ \Lambda(s \circ \Lambda(\bullet \circ (\mathsf{I},\bullet)))
  = s \circ \Lambda(\mathsf{B}_0 \circ (\mathsf{I} ,1))
  = s \circ \Lambda(\mathsf{B}_0) \circ \mathsf{I}
  = \mathsf{B}_1 \circ \mathsf{I}
  \\&= \bullet \circ (\mathsf{B},\mathsf{I})
  = (\vdash \mathsf{BI})
\end{align*}
For {\bf [P.6]}, we have:
\begin{align*}
  & (\vdash \mathsf{B}a^{\bullet\bullet}\mathsf{B})
    = \bullet \circ (\bullet \circ (\mathsf{B},(a^\bullet)^\bullet),\mathsf{B})
    = \bullet \circ (\mathsf{B}_1 \circ (a^\bullet)^\bullet,\mathsf{B})
    = \mathsf{B}_0 \circ ((a^\bullet)^\bullet,\mathsf{B})
  \\&= s \circ \Lambda(\bullet \circ ((a^\bullet)^\bullet,\bullet \circ (\mathsf{B},1)))
  = s \circ \Lambda(\bullet \circ ((a^\bullet)^\bullet,\mathsf{B}_1))
  = s \circ \Lambda(\bullet \circ (\mathsf{B}_1,a^\bullet))
  \\&= s \circ \Lambda(\mathsf{B}_0 \circ (1,a^\bullet))
  = s \circ \Lambda(s \circ \Lambda( \bullet \circ (1,\bullet \circ (a^\bullet,1))))
  \\&= s \circ \Lambda(s \circ \Lambda( \bullet \circ (1,\bullet \circ (1,a))))
  = s \circ \Lambda(s \circ \Lambda( \bullet \circ (\mathsf{B}_0,1) \circ (1,1,a)))
  \\&= s \circ \Lambda(s \circ \Lambda( \bullet \circ (1,a) \circ \mathsf{B}_0))
  = s \circ \Lambda(s \circ \Lambda( \bullet \circ (a^\bullet,1) \circ \mathsf{B}_0))
  \\&= s \circ \Lambda(s \circ \Lambda( \bullet \circ (a^\bullet,1) \circ \bullet \circ (\mathsf{B}_1,1)))
  = s \circ \Lambda(s \circ \Lambda( \bullet \circ (1,\bullet) \circ (a^\bullet,\mathsf{B}_1,1)))
  \\&= s \circ \Lambda(s \circ \Lambda(\bullet \circ (1,\bullet)) \circ (a^\bullet,\mathsf{B}_1))
  = s \circ \Lambda(\mathsf{B_0} \circ (a^\bullet,\mathsf{B}_1))
  \\&= s \circ \Lambda(\bullet \circ (\mathsf{B}_1 \circ a^\bullet,\bullet \circ (\mathsf{B},1)))
  = s \circ \Lambda(\bullet \circ (1,\bullet) \circ (\mathsf{B}_1 \circ a^\bullet,\mathsf{B},1))
  \\&= s \circ \Lambda(\bullet \circ (1,\bullet)) \circ (\mathsf{B}_1 \circ a^\bullet,\mathsf{B})
  = \mathsf{B}_0 \circ (\mathsf{B}_1 \circ a^\bullet,\mathsf{B})
  = \bullet \circ (\mathsf{B}_1 \circ \mathsf{B}_1 \circ a^\bullet,\mathsf{B})
  \\&= \bullet \circ (\bullet \circ (\mathsf{B},\mathsf{B}_1 \circ a^\bullet),\mathsf{B})
  = \bullet \circ (\bullet \circ (\mathsf{B},\bullet \circ (\mathsf{B}, a^\bullet)),\mathsf{B})
  = (\vdash \mathsf{B}(\mathsf{B}a^\bullet)\mathsf{B})
\end{align*}
For {\bf [P.7]} we have:
\begin{align*}
  & (\vdash \mathsf{B}b^\bullet(\mathsf{B}a^\bullet\mathsf{B}))
    = \bullet \circ (\bullet \circ (\mathsf{B},b^\bullet), \bullet \circ (\bullet \circ (\mathsf{B},a^\bullet), \mathsf{B}))
    = \bullet \circ (\mathsf{B_1} \circ b^\bullet, \bullet \circ (\mathsf{B}_1 \circ a^\bullet,\mathsf{B}))
  \\&= \mathsf{B}_0 \circ (b^\bullet, \mathsf{B}_0 \circ (a^\bullet, \mathsf{B}))
  = s \circ \Lambda(\bullet \circ (b^\bullet,\bullet \circ (\mathsf{B}_0 \circ (a^\bullet, \mathsf{B}),1)))
  \\&= s \circ \Lambda(\bullet \circ (b^\bullet, \bullet \circ (a^\bullet, \bullet \circ (\mathsf{B},1))   ))
  = s \circ \Lambda(\bullet \circ (b^\bullet, \bullet \circ (a^\bullet, \mathsf{B}_1)   ))
  \\&= s \circ \Lambda(\bullet \circ (b^\bullet, \bullet \circ (\mathsf{B}_1,a)   ))
  = s \circ \Lambda(\bullet \circ (b^\bullet,  \mathsf{B}_0 \circ (1,a)   ))
  = s \circ \Lambda(\bullet \circ (\mathsf{B}_0 \circ (1,a),b)  )
  \\&= s \circ \Lambda(\bullet \circ (\mathsf{B}_0,1) \circ (1,a,b))
  = s \circ \Lambda(\bullet \circ (1,\bullet) \circ (1,a,b))
  = s \circ \Lambda(\bullet \circ (1,\bullet \circ (a,b)))
  \\&= (\bullet \circ (a,b))^\bullet
  = (\vdash (ab)^\bullet)
\end{align*}
We know that {\bf [P.3]} follows from the other axioms, and so $(U,\bullet)$ is a planar $\lambda$-algebra in $\mathcal{M}$.

In order to complete the proof that $(U,\bullet)$ is a planar $\lambda$-model in $\mathcal{M}$, we need only show that is is weakly multi-extensional. To that end, we establish the following useful fact: for all $n \in \mathbb{N}$ we have: 
  \begin{equation}\label{eqn:bi-fact}
    (\vdash \mathsf{B}^{*n}\mathsf{I}) = (s \circ \Lambda)^{\circ n+1}(\bullet^n) \in \mathcal{M}(;A)
  \end{equation}
  where $(s \circ \Lambda)^{\circ n}(x)$ denotes the (syntactic) result of applying $s \circ \Lambda(-)$ to $x$ $n$ times, so that for example, $(s \circ \Lambda)^{\circ 3}(x) = s \circ \Lambda(s \circ \Lambda(s \circ \Lambda(x)))$. We prove this by induction on $n$. For the base case we have:
  \begin{align*}
    & (\vdash \mathsf{B}^{*0}\mathsf{I})
      = \mathsf{I}
      = s \circ \Lambda(1_U)
      = s \circ \Lambda(\bullet^0)
      = (s \circ \Lambda)^{\circ 1}(\bullet^0)
  \end{align*}
For the inductive case, if  $(\vdash \mathsf{B}^{*n}\mathsf{I}) = (s \circ \Lambda)^{\circ n+1}(\bullet^n)$ then we have:
\begin{align*}
  & (\vdash \mathsf{B}^{*n+1}\mathsf{I})
    = (\vdash \mathsf{B}(\mathsf{B}^{*n}\mathsf{I}))
    = \bullet \circ (\mathsf{B},\bullet \circ (\mathsf{B}^{*n},\mathsf{I}))
    = \mathsf{B}_1 \circ \bullet \circ (\mathsf{B}^{*n},\mathsf{I})
  \\&= s \circ \Lambda(\mathsf{B}_0) \circ \bullet \circ (\mathsf{B}^{*n},\mathsf{I})
  = s \circ \Lambda(s \circ \Lambda(\bullet \circ (1_U,\bullet))) \circ \bullet \circ (\mathsf{B}^{*n},\mathsf{I})
  \\&= s \circ \Lambda(s \circ \Lambda(\bullet \circ (1_U,\bullet) \circ (\bullet \circ (\mathsf{B}^{*n},\mathsf{I}),1_U,1_U)))
  \\&= s \circ \Lambda(s \circ \Lambda(\bullet \circ (\bullet \circ (\mathsf{B}^{*n},\mathsf{I}),\bullet)))
  = s \circ \Lambda(s \circ \Lambda(\bullet \circ ((s \circ \Lambda)^{\circ n+1}(\bullet^n),\bullet)))
  \\&= s \circ \Lambda(s \circ \Lambda((s \circ \Lambda)^{\circ n}(\bullet^n) \circ \bullet))
  = s \circ \Lambda(s \circ \Lambda((s \circ \Lambda)^{\circ n}(\bullet^n \circ (\bullet,1_U,\ldots,1_U))))
  \\&= (s \circ \Lambda)^{\circ n+2}(\bullet^{n+1})
\end{align*}
From which we conclude that the equation (\ref{eqn:bi-fact}) holds for all $n \in \mathbb{N}$.

We proceed to show that our applicative system is weakly multi-extensional. Suppose that we have $\bullet^n \circ (a,1_U,\ldots,1_U) = \bullet^n \circ (b,1_U,\ldots,1_U)$. Then we have:
\begin{align*}
  & (\vdash \mathsf{B}^{*n}\mathsf{I}a)
    = \bullet \circ (\bullet \circ (\mathsf{B}^{*n},\mathsf{I}),a)
    \overset{\text{(\ref{eqn:bi-fact})}}{=} \bullet \circ ((s \circ \Lambda)^{\circ n+1}(\bullet^n),a)
  \\&= (s \circ \Lambda)^{\circ n}(\bullet^n) \circ a
  = (s \circ \Lambda)^{\circ n}(\bullet^n \circ (a,1_U,\ldots,1_U))
  \\&= (s \circ \Lambda)^{\circ n}(\bullet^n \circ (b,1_U,\ldots,1_U))
  = (\vdash \mathsf{B}^{*n}\mathsf{I}b)
\end{align*}
and so $(A,\bullet)$ weakly multi-extensional, and is thus a planar $\lambda$-model.
\end{proof}

It now follows that the adjunction of Theorem~\ref{thm:w-triples-weakly-closed} holds for semi-closed multicategories and planar $\lambda$-models. We modify Definition~\ref{def:w-triples} to obtain:
\begin{definition}\label{def:s-triples}
  An \emph{$s$-triple} $(\mathcal{M},A,\bullet)$ consists of a multicategory $\mathcal{M}$ together with a planar $\lambda$-model $(A,\bullet)$ in $\mathcal{*}$. A \emph{morphism} of $s$-triples $F : (\mathcal{M},A,\bullet) \to (\mathcal{N},B,\odot)$ consists of a functor $F : \mathcal{M} \to \mathcal{N}$ such that $FA = B$ and $F(\bullet) = \odot$. Write $\mathsf{T}_s$ for the category of $s$-triples and their morphisms.
\end{definition}
Now it follows from Lemma~\ref{lem:semi-closed-reflexive} and Lemma~\ref{lem:computable-semi-closed} that the adjunction of Theorem~\ref{thm:w-triples-weakly-closed} specialises to planar $\lambda$-models and semi-closed multicategories, as in:
\begin{theorem}\label{thm:s-triples-semi-closed}
    There is an adjunction:
    \[
    \begin{tikzcd}
      \mathsf{O}_s \ar[r,phantom,"{\scriptstyle \perp}"] \ar[r,"L", shift left=0.5em] & \mathsf{T}_s \ar[l,"R", shift left=0.5em]
    \end{tikzcd}
  \]
  where $\mathsf{O}_s$ is the category of semi-closed operads and closed functors. Moreover, the unit of this adjunction is an isomorphism, witnessing $\mathsf{O}_s$ as a full subcategory of $\mathsf{T}_s$.
\end{theorem}

What this means is that the only real difference between a planar $\lambda$-model and a semi-closed multicategory is that while a planar $\lambda$-model must exist in some ambient multicategory, the corresponding semi-closed operad exists at the same metatheoretical level as that ambient multicategory. Given a planar $\lambda$-model $(A,\bullet)$ in $\mathcal{M}$, we can also think of $R(\mathcal{M},A,\bullet)$ as the smallest possible ambient multicategory that $(A,\bullet)$ can sensibly inhabit, wherein it appears as the induced planar $\lambda$-model $(*,\mathsf{ev})$.

\subsection{On Weak Multi-Extensionality and Having Enough Points}\label{sec:enough-points}
Here we reconcile the difference between the notion of weak multi-extensionality used in Definition~\ref{def:planar-lambda-model} and the (clearly weaker) notion of weak extensionality typically encountered in definitions of $\lambda$-model (see e.g.,~\cite{Scott1980,Hindley2008}). We show that whenever the ambient multicategory has enough points, the two notions coincide. In particular, the multicategory of sets and functions has enough points, and so in the classical setting it does not matter which we one we use.

We begin with the more traditional notion of weak extensionality:
\begin{definition}\label{def:weak-extensionality}
  Say that a planar $\lambda$-algebra is \emph{weakly extensional} in case for all $a,b \in \mathcal{M}(;A)$ we have:
  \[
    (x \vdash ax = bx) \Rightarrow (\vdash \mathsf{BI}a = \mathsf{BI}b)
  \]
\end{definition}

Next, we say what it means for a multicategory to have enough points:
\begin{definition}\label{def:enough-points}
  Say that a multicategory $\mathcal{M}$ has \emph{enough points} in case for all $f,g \in \mathcal{M}(A_1,\ldots,A_n ; B)$, if $f \circ (x_1,\ldots,x_n) = g \circ (x_1,\ldots,x_n)$ for all $(x_i \in \mathcal{M}(;A_i))_{i=1}^n$ then $f = g$.
\end{definition}

The promised result is now relatively straightforward:
\begin{lemma}
  Let $(A,\bullet)$ be a planar $\lambda$-algebra in a multicategory $\mathcal{M}$ with enough points. Then $(A,\bullet)$ is weakly extensional if and only if it is a planar $\lambda$-model.
\end{lemma}
\begin{proof}
  Clearly weak multi-extensionality implies weak extensionality, so we need only show the converse. To that end, suppose that $(A,\bullet)$ is weakly extensional. We proceed by induction on $n \in \mathbb{N}$. For the base case, if $\vdash a = b$ then immediately $\vdash \mathsf{B}^{*0}\mathsf{I}a = a = b = \mathsf{B}^{*0}\mathsf{I}b$. For the inductive case, suppose that we have:
  \[
    x,x_1,\ldots,x_n \vdash axx_1 \cdots x_n = bxx_1 \cdots x_n
  \]
  Then for all $c \in \mathcal{M}(;A)$ we have:
  \[
    x_1,\ldots,x_n \vdash acx_1 \cdots x_n = bcx_1 \cdots x_n
  \]
  and so by the inductive hypothesis we have $\vdash \mathsf{B}^{*n}\mathsf{I}(ac) = \mathsf{B}^{*n}\mathsf{I}(bc)$. This gives:
  \[
    \vdash \mathsf{B}^{*n+1}\mathsf{I}ac = \mathsf{B}(\mathsf{B}^{*n}\mathsf{I})ac = \mathsf{B}^{*n}\mathsf{I}(ac) = \mathsf{B}^{*n}\mathsf{I}(bc) = \mathsf{B}^{*n+1}\mathsf{I}bc
  \]
  Now since $c \in \mathcal{M}(;A)$ is arbitrary and $\mathcal{M}$ has enough points, we obtain:
  \[
    x \vdash \mathsf{B}^{*n+1}\mathsf{I}ax = \mathsf{B}^{*n+1}\mathsf{I}bx
  \]
  but then weak extensionality gives $\vdash \mathsf{BI}(\mathsf{B}^{*n+1}\mathsf{I}a) = \mathsf{BI}(\mathsf{B}^{*n+1}\mathsf{I}b)$ and we have: 
  \begin{align*}
    & \vdash \mathsf{B}^{*n+1}\mathsf{I}a
      = \mathsf{B}(\mathsf{B}^{*n}\mathsf{I})a
      \overset{\text{\ref{lem:lambda-algebra-facts}.(vi)}}{=} \mathsf{BI}(\mathsf{B}(\mathsf{B}^{*n}\mathsf{I})a)
      \\&= \mathsf{BI}(\mathsf{B}^{*n+1}\mathsf{I}a)
      = \mathsf{BI}(\mathsf{B}^{*n+1}\mathsf{I}b)
      = \mathsf{B}^{*n+1}\mathsf{I}b
  \end{align*}
  The claim follows.
\end{proof}
  
\section{A Planar Version of Scott's Representation Theorem}\label{sec:planar-scott}
In this section we obtain a planar version of Scott's representation theorem~\cite{Scott1980}. That is, we show that every model of the planar $\lambda$-calculus (i.e., semi-closed multicategory) arises from a reflexive object in a closed multicategory. As an intermediate step, we show that the idempotent splitting completion of any semi-closed multicategory is closed. To begin, we define:
\begin{definition}\label{def:multi-idempotent-splitting}
  Let $\mathcal{M}$ be a multicategory. The \emph{idempotent splitting of $\mathcal{M}$}, written $\mathsf{Split}(\mathcal{M})$, is the multicategory whose objects are pairs $(A,a)$ consisting of an object $A \in \mathcal{M}_0$ and an idempotent $a = a \circ a \in \mathcal{M}(A;A)$. Morphisms $f \in \mathsf{Split}(\mathcal{M})((A_1,a_1),\ldots,(A_n,a_n);(B,b))$ are morphisms $f \in \mathcal{M}(A_1,\ldots,A_n;B)$ such that $b \circ f \circ (a_1,\ldots,a_n) = f$. For identities we take $1_{(A,a)} = a$, and composition is given by composition in $\mathcal{M}$. 
\end{definition}

Next, in any semi-closed multicategory $\mathcal{M}$, for $f \in \mathcal{M}(A;B)$, $g \in \mathcal{M}(C;D)$, we define a morphism $[f,g] \in \mathcal{M}([B,C];[A,D])$ as in:
\[
  [f,g] = \Lambda^{[B,C]}_{A,D}(g \circ \mathsf{ev}_{B,C} \circ (1_{[B,C]},f))
\]
We now have:
\begin{lemma}\label{lem:split-closed}
  Let $\mathcal{M}$ be a semi-closed multicategory. Then $\mathsf{Split}(\mathcal{M})$ is a closed multicategory.
\end{lemma}
\begin{proof}
  For $(A,a)$ and $(B,b)$ objects of $\mathsf{Split}(\mathcal{M})$ we define $[(A,a),(B,b)] = ([A,B],[a,b])$. For this to be well-defined we must have  that $[a,b]$ is idempotent. Indeed, we have:
  \begin{align*}
    & [a,b] \circ [a,b]
      = \Lambda^{[A,B]}_{A,B}(b \circ \mathsf{ev}_{A,B} \circ (1_{[A,B]},a)) \circ \Lambda^{[A,B]}_{A,B}(b \circ \mathsf{ev}_{A,B} \circ (1_{[A,B]},a))
    \\&= \Lambda^{[A,B]}_{A,B}(b \circ \mathsf{ev}_{A,B} \circ (1_{[A,B]},a) \circ ( \Lambda^{[A,B]}_{A,B}(b \circ \mathsf{ev}_{A,B} \circ (1_{[A,B]},a)),1_A) )
    \\&= \Lambda^{[A,B]}_{A,B}(b \circ \mathsf{ev}_{A,B} \circ ( \Lambda^{[A,B]}_{A,B}(b \circ \mathsf{ev}_{A,B} \circ (1_{[A,B]},a)),1_A)  \circ (1_{[A,B]},a) )
    \\&= \Lambda^{[A,B]}_{A,B}(b \circ b \circ \mathsf{ev}_{A,B} \circ (1_{[A,B]},a)\circ (1_{[A,B]},a))
    = \Lambda^{[A,B]}_{A,B}(b \circ \mathsf{ev}_{A,B} \circ (1_{[A,B]},a))
    = [a,b]
  \end{align*}
  Now $\mathsf{ev}_{(A,a),(B,b)} \in \mathsf{Split}(\mathcal{M})([(A,a),(B,b)],(A,a) ; (B,b))$ is given by $b \circ \mathsf{ev}_{A,B} \circ ([a,b],a) \in \mathcal{M}([A,B],A;B)$. Clearly this is well-defined as a map of $\mathsf{Split}(\mathcal{M})$. Given a morphism $f \in \mathsf{Split}(\mathcal{M})(\mathbf{\Gamma},(A,a);(B,b))$ where $\mathbf{\Gamma} = (C_1,c_1),\ldots,(C_n,c_n)$, write $\Gamma = C_1,\ldots,C_n$ and define $\Lambda^\mathbf{\Gamma}_{(A,a),(B,b)}(f) \in \mathsf{Split}(\mathcal{M})(\mathbf{\Gamma};[(A,a),(B,b)])$ to be the morphism given by $\Lambda^\Gamma(f) \in \mathcal{M}(\Gamma;[A,B])$. We show that this is well-defined:
  \begin{align*}
    & [a,b] \circ \Lambda^\Gamma_{A,B}(f) \circ (c_1,\ldots,c_n)
      = \Lambda^{[A,B]}_{A,B}(b \circ \mathsf{ev}_{A,B} \circ (1_{[A,B]},a)) \circ \Lambda^\Gamma_{A,B}(f \circ (c_1,\ldots,c_n,1_A))
    \\&= \Lambda^{\Gamma}_{A,B}(b \circ \mathsf{ev}_{A,B} \circ (1_{[A,B]},a) \circ (\Lambda^\Gamma_{A,B}(f \circ (c_1,\ldots,c_n,1_A)),1_A))
    \\&= \Lambda^{\Gamma}_{A,B}(b \circ \mathsf{ev}_{A,B} \circ (\Lambda^\Gamma_{A,B}(f \circ (c_1,\ldots,c_n,1_A)),1_A)  \circ (1_{[A,B]},a))
    \\&= \Lambda^\Gamma_{A,B}(b \circ f \circ (c_1,\ldots,c_n,a))
    = \Lambda^\Gamma_{A,B}(f)
  \end{align*}
  We have $\mathsf{ev}_{(A,a),(B,b)} \circ (\Lambda^\mathbf{\Gamma}_{(A,a),(B,b)}(f),1_{(A,a)}) = f$ as in:
  \begin{align*}
    & b \circ \mathsf{ev}_{A,B} \circ ([a,b],a) \circ (\Lambda^\Gamma_{A,B}(f),a)
      = b \circ \mathsf{ev}_{A,B} \circ ([a,b] \circ \Lambda^\Gamma_{A,B}(f), a \circ a)
      \\&= b \circ \mathsf{ev}_{A,B} (\Lambda^\Gamma_{A,B}(f),1_A) \circ (1_\Gamma,a)
    = b \circ f \circ (1_{\Gamma},a)
    = f
  \end{align*}
  It remains to show that for any $h \in \mathsf{Split}(\mathcal{M})(\mathbf{\Gamma};[(A,a),(B,b)])$, the equation $\Lambda^{\mathbf{\Gamma}}_{(A,a),(B,b)}(\mathsf{ev}_{(A,a),(B,b)} \circ (h,1_{(A,a)})) = h$ holds. Indeed, we have:
  \begin{align*}
    & \Lambda^\Gamma_{A,B}(b \circ \mathsf{ev}_{A,B} \circ ([a,b],a) \circ (h,a))
    \\& = \Lambda^{[A,B]}_{A,B}(b \circ \mathsf{ev}_{A,B} \circ (\Lambda^{[A,B]}_{A,B}(b \circ \mathsf{ev}_{A,B} \circ (1_{[A,B]},a)),1_A) \circ (1_{[A,B]},a)) \circ h
    \\&= \Lambda^{[A,B]}_{A,B}(b \circ b \circ \mathsf{ev}_{A,B} \circ (1_{[A,B]},a) \circ (1_{[A,B]},a)) \circ h
    \\&=  \Lambda^{[A,B]}_{A,B}(b \circ \mathsf{ev}_{A,B} \circ (1_{[A,B]},a)) \circ h
    = [a,b]\circ h = h
  \end{align*}
  Thus, $\mathsf{Split}(\mathcal{M})$ is closed.
\end{proof}

The promised analogue of Scott's representation theorem is now:
\begin{theorem}\label{thm:planar-scott}
  Let $\mathcal{M}$ be a semi-closed operad. Then $((*,1),\mathsf{ev})$ is a planar $\lambda$-model in $\mathsf{Split}(\mathcal{M})$ and there is an isomorphism of operads:
  \[
    \mathcal{M} \cong R(\mathsf{Split}(\mathcal{M}),(*,1),\mathsf{ev})
  \]
\end{theorem}
\begin{proof}
  That $(*,1)$ is a planar $\lambda$-model in $\mathsf{Split}(\mathcal{M})$ follows from the fact that it is a planar $\lambda$-model in $\mathcal{M}$ (Lemma~\ref{lem:semi-closed-reflexive}). A morphism $\overline{f} \in R(\mathsf{Split}(\mathcal{M}),(*,1),\mathsf{ev})$ is precisely a $(*,1)$-computable morphism $f \in \mathsf{Split}(\mathcal{M})((*,1)^n;(*,1))$, which is precisely a morphism $f \in \mathcal{M}(n)$. The claim follows.
\end{proof}

In particular, $((*,1),[1,1] : (*,[1,1]) \to (*,1), [1,1] : (*,1) \to (*,[1,1]))$ is always a reflexive object in $\mathsf{Split}(\mathcal{M})$. For any semi-closed operad $\mathcal{M}$, Lemma~\ref{lem:split-closed} tells us that $\mathsf{Split}(\mathcal{M})$ is closed, and so Theorem~\ref{thm:planar-scott} tells us that $\mathcal{M}$ arises from a reflexive object in a closed multicategory. Since semi-closed operads are precisely models of the planar $\lambda$-calculus, this can be understood as a (rather slick) version of Scott's representation theorem.

\section{Concluding Remarks and Future Work}\label{sec:conclusion}
We have rigorously compared two notions of model of the planar $\lambda$-calculus: planar $\lambda$-models and semi-closed multicategories. Specifically, we have constructed an adjunction capturing the relationship between these two notions of model (Theorem~\ref{thm:s-triples-semi-closed}). We have moreover used this machinery to obtain a version of Scott's representation theorem for the planar $\lambda$-calculus (Theorem~\ref{thm:planar-scott}).

We view these results as the first step in the larger project of systematically relating different sorts of semi-closed multicategory to different notions of \emph{substructural} $\lambda$-model, much as in the work of Kuzmin et. al., in which different sorts of weakly-closed multicategory are related to different notions of substructural combinatory algebra~\cite{Kuzmin2026Journal}. Ultimately, the aim is to obtain a systematic understanding of the relationship between the combinatory approach to modelling various substructural $\lambda$-calculi and the operadic approach of Hyland~\cite{Hyland2017}.

\bibliographystyle{plain}
\bibliography{citations.bib}

\end{document}